\documentclass[a4paper,11pt,twoside]{article}

\usepackage{geometry}
\usepackage[utf8]{inputenc}
\usepackage{lmodern}
\usepackage[T1]{fontenc}
\usepackage{amsmath,amssymb,amsthm,empheq,cases}
\usepackage{hyperref}
\hypersetup{colorlinks,
            citecolor=red, 
            filecolor=black,
            linkcolor=blue,
            urlcolor=black}
\usepackage{tikz}
\usepackage{graphicx}
\usepackage{url}

\def \dis {\displaystyle}

\def \NN {\mathbb N}

\def \RR {\mathbb R}

\def \A {\mathcal{A}}
\def \B {\mathcal{B}}

\def \F {\mathcal{F}}

\def \L {\mathcal{L}}

\def \P {\mathcal{P}}

\def \ecart {\noalign{\medskip}}

\theoremstyle{definition}
\newtheorem{Th}{Theorem}[section]

\newtheorem{Lem}[Th]{Lemma}
\newtheorem{Cor}[Th]{Corollary}

\newtheorem{Rem}[Th]{Remark}
 
\def \refs #1{Section~\ref{#1}}

\def \refT #1{Theorem~\ref{#1}}

\title{Solvability of a fourth order elliptic problem \\ in a bounded sector, part I}
\author{Rabah Labbas, St\'ephane Maingot \& Alexandre Thorel \\ \ecart
\scriptsize R. L., S. M. \& A. T.: Normandie Univ, UNIHAVRE, LMAH, FR-CNRS-3335, ISCN, 76600 Le Havre, France. \\ \ecart 
\scriptsize rabah.labbas@univ-lehavre.fr, stephane.maingot@univ-lehavre.fr, alexandre.thorel@univ-lehavre.fr}
\date{}

\begin{document}

\maketitle

\begin{abstract}
The purpose of this article (composed of two parts) is the study of the generalized dispersal operator of a reaction-diffusion equation in $L^p$-spaces set in the finite conical domain $S_{\omega,\rho}$ of angle $\omega>0$ and radius $\rho>0$ in $\RR^2$. 

This first part is devoted to the behavior of the solution near the top of the cone which is completely described in the weighted Sobolev space $W^{4,p}_{3-\frac{1}{p}}(S_{\omega,\rho_0})$, $\rho_0 \leqslant \rho$, see \refT{Th principal}. 
\end{abstract}

\section{Introduction}

This work required the use of many calculations and non-trivial checks linked, among other, to the theory of sums of linear operators. This is why we were forced to split this work into two more or less independent parts.

We consider the following bounded conical domain
\begin{equation*}
S_{\omega,\rho }=\left\{ (x,y)=(r\cos \theta ,r\sin \theta ):0<r<\rho \text{ and } 0<\theta <\omega \right\},
\end{equation*}
and its three edges: 
\begin{equation*}
\left\{ 
\begin{array}{lll}
\Gamma_{0} &=& (0,\rho)  \times \left\{ 0\right\} \\ 
\Gamma_{\omega} & = & \displaystyle\left\{ (r\cos \omega ,r\sin \omega )~:~0<r<\rho
\right\} \\
\Gamma_\rho & = & \displaystyle \left\{ (\rho\cos \theta ,\rho\sin \theta )~:~0<\theta<\omega\right\},
\end{array}\right.
\end{equation*}
where $\rho > 0$ is given and $\omega \in (0, 2\pi]$.

Let $k$ be a positive number, $f^*$ be a non-linear reaction function, $u_0$ be a given function and $n$ be the outwards normal unit vector to $\partial S_{\omega,\rho}$. 

We recall that the resolution of the generalized following reaction-diffusion problem
\begin{equation}\label{eq evolution}
\left\{\begin{array}{cll}
\dis\frac{\partial u}{\partial t} (t,x,y) & = & -\Delta^2_{(x,y)} u (t,x,y) + k \Delta_{(x,y)} u (t,x,y) + f^*(u(t,x,y)) \quad \text{in} ~ \RR_+\times S_{\omega,\rho} \\ \ecart
u(0,x,y) &=& u_0(x,y), ~(x,y) \in S_{\omega,\rho} \\ \ecart
u(t,\sigma) &=& \dfrac{\partial u}{\partial n} (t,\sigma) = 0, ~ (t,\sigma) \in \RR_+ \times \Gamma _{0}\cup \Gamma_{\omega } \\ \ecart
u(t,\sigma) &=& \dfrac{\partial^2 u}{\partial n^2} (t,\sigma) = 0, ~ (t,\sigma) \in \RR_+ \times \Gamma _{\rho},
\end{array}\right.
\end{equation}
needs, in a first step, to solve the following linear stationary problem 
\begin{equation}\label{Pb cone infini}
\left\{ \begin{array}{ll}
\Delta ^{2}u-k\Delta u=f &\in L^p(S_{\omega,\rho }),\quad p \in (1,+ \infty)  \\ \ecart 
u=\dfrac{\partial u}{\partial n}=0 & \text{on }\Gamma _{0}\cup \Gamma
_{\omega } \\ \ecart
u = \dfrac{\partial^2 u}{\partial n^2} = 0 & \text{on } \Gamma _{\rho}.
\end{array}\right. 
\end{equation}
This first study will allow us to determine, by an analogous calculus, the resolvent of the corresponding operator by solving the spectral problem 
\begin{equation*}
\left\{ \begin{array}{ll}
\Delta ^{2}u-k\Delta u - \lambda u = f & \in L^p(S_{\omega,\rho }) \\ \ecart 
u=\dfrac{\partial u}{\partial n}=0 & \text{on }\Gamma _{0}\cup \Gamma
_{\omega } \\ \ecart
u = \dfrac{\partial^2 u}{\partial n^2} = 0 & \text{on } \Gamma _{\rho},
\end{array}\right.
\end{equation*}
and then to estimate the resolvent operator in view to obtain the generation of an analytic semigroup as in Labbas, Maingot and Thorel \cite{LMT}. In the last step, the fixed point method can be applied to solve \eqref{eq evolution}.

In this first part, we will focus ourselves on the study of problem \eqref{Pb cone infini}. This work is a natural continuation of the one studied in Labbas, Maingot, Manceau and Thorel \cite{LMMT}. 

The originality of this work lies in the fact that the open set $S_{\omega,\rho}$ is conical whereas in the cited works above, it was cylindrical and in the fact that the spatial operator is composed by a linear combination of a laplacian and a bilaplacian operators.

Such problems, set in conical domains, model many concrete situations related to pollution for instance.

After some variables and functions changes, problem \eqref{Pb cone infini} will be written as a abstract sum of unbounded linear operators in a Banach space, see \eqref{Pb L1+L2}. Therefore, we will use the sum theory carried out in part II to solve \eqref{Pb L1+L2}, see Labbas, Maingot and Thorel \cite{Cone P2}.

This work is inspired by the one done by G.~Geymonat and P.~Grisvard in \cite{geymonat-grisvard} and P.~Grisvard in \cite{grisvard}, where the authors have considered, in a hilbertian framework, the following boundary problem
\begin{equation*}\label{Pb Grisvard}
\left\{ \begin{array}{ll}
\Delta^2 u = 0 &\text{in }S_{\omega } \\ \ecart 
u=\dfrac{\partial u}{\partial n}=0 & \text{on }\Gamma _{0}\cup \Gamma
_{\omega },
\end{array}\right. 
\end{equation*}
where $S_\omega$ is the infinite conical domain 
$$S_{\omega}=\left\{ (x,y)=(r\cos \theta ,r\sin \theta ):r>0 \text{ and } 0<\theta <\omega \right\}.$$

The author has proved that the solution of this problem is written as a "superposition" of particular solutions with separate variables of the form $\chi_{1,j}(r)\chi_{2,j}(\theta)$, for $j \in \NN$. The basic tools used are based on the compact operators belonging to the so-called Carleman class and the Fredholm determinants. 

Let us recall some known results concerning the biharmonic equation in a conical domain or in a Lipschitz domain. In Pipher and Verchota \cite{pipher}, the authors gave many estimates concerning the solution of the Dirichlet problem in $L^p$ for the biharmonic equation in Lipschitz domain. In Costabel, Dauge and Nicaise \cite{costabel}, the authors have used Mellin transformation in view to give an optimal characterization of the structure of weighted Sobolev spaces with nonhomogeneous norms on finite cones. In Barton and Mayboroda \cite{barton}, many results are given for general higher-order elliptic equations in non smooth-domains. In Tami \cite{tami}, the author has studied the following problem
\begin{equation*}
\left\{ \begin{array}{ll}
\Delta^2 u = f &\text{in }S_{\omega,1} \\ \ecart 
u=\Delta u =0 & \text{on } \partial S_{\omega,1},
\end{array}\right. 
\end{equation*}
where $ f \in L^2(S_{\omega,1})$. He has proved the two following results
\begin{enumerate}
\item If $\omega < \pi$, the variational solution is written, in the neighbourhood of the corner, as 
$$u_\omega = u_{1,\omega} + u_{2,\omega} + u_{3,\omega},$$
with $u_{1,\omega} \in H^{1+\frac{\pi}{\omega}-\varepsilon}$, $u_{2,\omega} \in H^{2+\frac{\pi}{\omega}-\varepsilon}$ and $u_{3,\omega} \in H^{4}$, for a small $\varepsilon > 0$.  

\item If $\omega = \pi$, in the neighbourhood of the corner, the variational solution verifies  
$$u_\pi \in H^4.$$
\end{enumerate}

This article is organized as follows. In \refs{Sect Statement of result}, we state our problem in polar coordinates and our main result in \refT{Th principal}. In \refs{Sect Reformulation} we reformulate our problem as a sum of linear unbounded operators. Then, \refs{Sect proof of main Th} is devoted to the proof of \refT{Th principal}.

\section{Statement of the main result}\label{Sect Statement of result}

We introduce the following polar variables function  
\begin{equation*}
v(r,\theta )=u(r\cos \theta ,r\sin \theta )=u(x,y).
\end{equation*}
It is known that the laplacian and the bilaplacian, in polar coordinates, are respectively written as
\begin{equation}\label{Lambda 1}
\Delta u = \frac{1}{r^{2}}\left[ \left( r\dfrac{\partial }{\partial r}\right)
^{2}+\dfrac{\partial ^{2}}{\partial \theta ^{2}}\right] v=\dfrac{\partial
^{2}v}{\partial r^{2}}+\frac{1}{r}\dfrac{\partial v}{\partial r}+\frac{1}{%
r^{2}}\dfrac{\partial ^{2}v}{\partial \theta ^{2}} := \Lambda_1 v,
\end{equation}
and 
\begin{equation}\label{Lambda 2}
\begin{array}{rcl}
\Delta ^{2}u &\hspace{-0.1cm}=&\hspace{-0.1cm} \dis\left( \dfrac{\partial ^{2}}{\partial r^{2}}+\frac{1}{r}\dfrac{\partial }{\partial r}+\frac{1}{r^{2}}\dfrac{\partial ^{2}}{\partial \theta^{2}}\right) ^{2}v \\ \ecart
&=&\hspace{-0.1cm}\dis \dfrac{\partial ^{4}v}{\partial r^{4}}+\frac{2}{r^{2}}\dfrac{\partial
^{4}v}{\partial r^{2}\partial \theta ^{2}} +\frac{1}{r^{4}}\dfrac{\partial ^{4}v}{\partial \theta ^{4}} + \frac{2}{r}\dfrac{\partial^{3}v}{\partial r^{3}} -\frac{2}{r^{3}}\dfrac{\partial ^{3}v}{\partial r\partial \theta ^{2}}-\frac{1}{r^{2}}\dfrac{\partial ^{2}v}{\partial r^{2}}+\frac{4}{r^{4}}\dfrac{\partial ^{2}v}{\partial \theta ^{2}}+\frac{1}{r^{3}}\dfrac{\partial v}{\partial r} \\ \ecart
&:=& \Lambda_2 v.
\end{array}
\end{equation}

\begin{Rem}
We can generalize this work to the dimension $n$ :
\begin{equation*}
\Delta u=\dfrac{\partial ^{2}v}{\partial r^{2}}+\frac{n-1}{r}\dfrac{\partial
v}{\partial r}+\frac{1}{r^{2}}\Delta ^{\prime }v,
\end{equation*}
where $\Delta ^{\prime }$ is the Laplace-Beltrami operator on the sphere.
\end{Rem}
We set
$$f(x,y)=f\left( r\cos \theta ,r\sin \theta\right) =g(r,\theta ).$$
Let $\gamma \in \RR$. 
In the case of the unbounded conical domain 
$$S_{\omega}=\left\{ (x,y)=(r\cos \theta ,r\sin \theta ):r>0 \text{ and } 0<\theta <\omega \right\},$$
we will use the following weighted Sobolev spaces of Kondrat'ev type
$$L^p_\gamma (S_{\omega}) = \left\{v \text{ measurable on }S_{\omega} : \int_{S_{\omega}}|v(r,\theta)|^p r^{p\gamma} ~dr \,d\theta < +\infty \right\},$$
and
\begin{equation}\label{Def Wmp generale}
W^{m,p}_\gamma (S_{\omega}) = \left\{v \in L^p_\gamma(S_{\omega}), \forall\, (i,j) \in \mathbb{N}^2 : 0 \leqslant i+j \leqslant m, ~ r^{\gamma+i+j}\frac{\partial^{i+j} v}{\partial r^i \partial \theta^j} \in L^p(S_{\omega}) \right\},
\end{equation}
see, for instance, Definition 2.1 in Costabel, Dauge and Nicaise \cite{costabel}.

In the sequel, the definitions above, for the bounded conical domain $S_{\omega,\rho}$, coincide with 
$$L^p_\gamma (S_{\omega,\rho}) = \left\{v \text{ measurable on }S_{\omega,\rho} : \int_{S_{\omega,\rho}}|v(r,\theta)|^p r^{p\gamma} ~dr \,d\theta < +\infty \right\},$$
and
\begin{equation}\label{Def Wmp}
W^{m,p}_\gamma (S_{\omega,\rho}) = \left\{v \in L^p_\gamma(S_{\omega,\rho}), \forall\, (i,j) \in \mathbb{N}^2 : 0 \leqslant i+j \leqslant m, ~ r^\gamma \frac{\partial^{i+j} v}{\partial r^i \partial \theta^j} \in L^p(S_{\omega,\rho}) \right\},
\end{equation}
since clearly \eqref{Def Wmp} implies \eqref{Def Wmp generale} in the bounded conical domain $S_{\omega,\rho}$.

Note that $f \in L^p(S_{\omega,\rho})$ means that $g \in L^p_{\frac{1}{p}} (S_{\omega,\rho})$. In fact, we have
$$\int_{S_{\omega,\rho}} |f(x,y)|^p dx dy = \int_{S_{\omega,\rho}} |g(r,\theta)|^p r dr d\theta,$$
where
$$g(r,\theta ) = f\left( r\cos \theta ,r\sin \theta\right), \quad 0 < r < \rho,~ 0 < \theta < \omega.$$
Then, in polar coordinates, problem \eqref{Pb cone infini} is written in the following form 
\begin{equation}\label{pb v}
\left\{\begin{array}{l}
\Lambda_2 v - k \Lambda_1 v = g \quad \text{in } S_{\omega,\rho} \\ \ecart
v(r,0) = v(r,\omega ) = \dfrac{\partial v}{\partial \theta }(r,0) = \dfrac{\partial v}{\partial \theta }(r,\omega) = 0 \\ 
v(\rho,\theta) = \dfrac{\partial^2 v}{\partial r^2}(\rho,\theta) = 0,
\end{array}\right.
\end{equation}
where $\Lambda_1$ and $\Lambda_2$ are given by \eqref{Lambda 1} and \eqref{Lambda 2}.

In this article, we will focus ourselves on the resolution of problem \eqref{pb v} to obtain the behavior of the solution $v$ to problem \eqref{pb v} in $L^p$-weighted spaces, in the neighborhood of the top of the cone. To this $v$ corresponds a solution $u_0$ to problem \eqref{Pb cone infini} by applying the inverse changes of variables and functions.

To solve problem \eqref{pb v}, we will use results given in Labbas, Maingot and Thorel \cite{Cone P2} and to this end, we need to consider 
$$\tau = \min_{j\geqslant 1}\left|\text{Im}(z_j)\right| > 0,$$
where $(z_j)_{j\geqslant 1}$ are the solutions of the following transcendent equation
$$\left(\sinh(z) + z\right) \left(\sinh(z) - z\right) = 0, \quad \text{with} \quad\text{Re}(z)>0.$$
Actually, according to Fädle \cite{fadle}, we have
$$\tau \simeq 4.21239.$$
We assume that 
\begin{equation}\label{hyp inv sum 0}
\omega\left(3 - \frac{2}{p}\right) < \tau.
\end{equation}
Recall that $\dis 3-\frac{2}{p}$ is exactly the Sobolev exponent of the space $W^{3,p}$ in two variables. 

Let us remark that we have two cases: 
\begin{enumerate}
\item If $\dis 0 < \omega \leqslant \frac{\tau}{3} \simeq 0.45\pi$, then \eqref{hyp inv sum 0} is satisfied for all $p \in (1,+\infty)$.

\item If $\dis \frac{\tau}{3} < \omega < \tau \simeq 1.34\pi$, then \eqref{hyp inv sum 0} is satisfied for $\dis 1 < p < \frac{2\omega}{3\omega - \tau}$.
\end{enumerate}
Our main result is the following:
\begin{Th}\label{Th principal}
Let $g \in L^p_{\frac{1}{p}}(S_{\omega,\rho})$ with $p$ satisfying assumption \eqref{hyp inv sum 0}. Then, there exists $\rho_0 > 0$ such that problem \eqref{pb v} has a unique solution $v$ satisfying
$$v \in W^{4,p}_{3 - \frac{1}{p}}(S_{\omega,\rho_0}).$$
\end{Th}
\begin{Rem}
Note that, we can also verify that we have in addition: for $i=0,1,2$ and $j = 0,1,2,3,4$ such that $0\leqslant i+j\leqslant 4$ 
$$ \frac{\partial^{i+j} v}{\partial r^i \partial \theta^j} \in L^p(S_{\omega,\rho_0}).$$ 
In fact, we have 
$$\dfrac{\partial^3 v}{\partial r^3}, \dfrac{\partial^4 v}{\partial r^3 \partial \theta} \in L^p_{2-\frac{1}{p}}(S_{\omega,\rho_0})\quad \text{and} \quad \dfrac{\partial^4 v}{\partial r^4} \in L^p_{3-\frac{1}{p}}(S_{\omega,\rho_0}),$$ 
see \refs{Sect Reg V2}; but, these two weighted spaces are not embedded into the space $L^p(S_{\omega,\rho_0})$.
\end{Rem}

\section{Reformulations of problem \eqref{pb v}} \label{Sect Reformulation}

\subsection{Some preliminary calculus}

Let us introduce the auxiliary function $\displaystyle \frac{v}{r}$. Then 
\begin{eqnarray*}
\left( r\dfrac{\partial }{\partial r}\right) ^{2}\left( \frac{v}{r}\right) &=&\left( r\dfrac{\partial }{\partial r}\right) \left( r\dfrac{\partial }{%
\partial r}\right) \left(\frac{v}{r}\right) \\ 
&=&\frac{v}{r}-\dfrac{\partial v}{\partial r}+r\dfrac{\partial ^{2}v}{\partial r^{2}}
\\
&=&\frac{v}{r}-2\dfrac{\partial }{\partial r}\left( r.\frac{v}{r}\right) +r\left( \frac{1}{r}%
\dfrac{\partial v}{\partial r}+\dfrac{\partial ^{2}v}{\partial r^{2}}\right) 
\\
&=& - \frac{v}{r} - 2r\dfrac{\partial }{\partial r}\left( \frac{v}{r}\right) +r\left( \frac{1}{r}%
\dfrac{\partial v}{\partial r}+\dfrac{\partial ^{2}v}{\partial r^{2}}\right),
\end{eqnarray*}
so
\begin{equation*}
\Delta u =\left( \frac{1}{r}\dfrac{\partial v}{\partial r}+\dfrac{\partial^{2} v}{\partial r^{2}}\right) +\frac{1}{r^{2}}\dfrac{\partial ^{2}v}{\partial \theta ^{2}} = \frac{1}{r}\left[ \left( r\dfrac{\partial }{\partial r}\right) ^{2}\left(\frac{v}{r}\right) + \frac{v}{r} + 2r\dfrac{\partial }{\partial r}\left( \frac{v}{r}\right) \right] + \frac{1}{r}\dfrac{\partial ^{2}}{\partial \theta ^{2}}\left( \frac{v}{r}\right).
\end{equation*}
Moreover, we have
\begin{eqnarray*}
\Pi_1 & := &\dis \left( r\dfrac{\partial }{\partial r}\right) ^{2}\left[ \left( r\dfrac{\partial }{\partial r}\right) ^{2}\left[ \frac{v}{r}\right] \right] \\
 & = & \left( r\dfrac{\partial }{\partial r}\right) ^{2}\left[\left( r\dfrac{\partial }{\partial r}\right) 
\left( -\frac{v}{r}+\dfrac{\partial v}{\partial r}\right) \right] \\
 &=& \left( r\dfrac{\partial }{\partial r}\right) ^{2}\left[ \frac{v}{r}-%
\dfrac{\partial v}{\partial r}+r\dfrac{\partial ^{2}v}{\partial r^{2}}\right]
\\
&=&\left( r\dfrac{\partial }{\partial r}\right) \left[ -\frac{v}{r}+\dfrac{%
\partial v}{\partial r}+r^{2}\dfrac{\partial ^{3}v}{\partial r^{3}}\right] 
\\
&=&\frac{v}{r}-\dfrac{\partial v}{\partial r}+r\dfrac{\partial ^{2}v}{%
\partial r^{2}}+2r^{2}\dfrac{\partial ^{3}v}{\partial r^{3}}+r^{3}\dfrac{%
\partial ^{4}v}{\partial r^{4}},
\end{eqnarray*}
and
\begin{eqnarray*}
\Pi_2 & := & 2\left( \dfrac{\partial ^{2}}{\partial \theta ^{2}}-1\right) \left( r \dfrac{\partial }{\partial r}\right) ^{2}\left( \frac{v}{r}\right) +\left( \dfrac{\partial ^{2}}{\partial \theta ^{2}}+1\right) ^{2}\left( \frac{v}{r}\right)  \\ 
&=& 2 \left( \dfrac{\partial ^{2}}{\partial \theta ^{2}}-1\right) \left( \frac{v}{r} -
\dfrac{\partial v}{\partial r}+r\dfrac{\partial ^{2}v}{\partial r^{2}}%
\right) +\frac{1}{r}\left( \dfrac{\partial ^{2}}{\partial \theta ^{2}}%
+1\right) ^{2}v \\
&=&2\dfrac{\partial v}{\partial r}-2r\dfrac{\partial ^{2}v}{\partial r^{2}}+%
\frac{4}{r}\dfrac{\partial ^{2}v}{\partial \theta ^{2}}-2\dfrac{\partial
^{3}v}{\partial \theta ^{2}\partial r}+2r\dfrac{\partial ^{4}v}{\partial
\theta ^{2}\partial r^{2}}+\frac{1}{r}\dfrac{\partial ^{4}v}{\partial \theta
^{4}}-\frac{v}{r}.
\end{eqnarray*}
Therefore, we obtain that
$$r^3 \Delta^2 u = \Pi_1 + \Pi_2.$$
We set 
\begin{equation*}
w=\frac{v}{r}.
\end{equation*}
Then, in $S_{\omega,\rho}$, problem \eqref{pb v} becomes 
\begin{equation} \label{EquationEn RetTheta}
\left\{ 
\begin{array}{l}
\dfrac{1}{r^{3}}\left[ \left( r\dfrac{\partial }{\partial r}\right) ^{2}%
\left[ \left( r\dfrac{\partial }{\partial r}\right) ^{2}w\right] +2\left( 
\dfrac{\partial ^{2}}{\partial \theta ^{2}}-1\right) \left( r\dfrac{\partial 
}{\partial r}\right) ^{2}w+\left( \dfrac{\partial ^{2}}{\partial \theta ^{2}}%
+1\right) ^{2}w\right] \\ \ecart
-\dfrac{k}{r}\left[ \left( \left( r\dfrac{\partial }{\partial r}\right)
^{2}w+2\left( r\dfrac{\partial }{\partial r}\right) w+w\right) +\dfrac{%
\partial ^{2}w}{\partial \theta ^{2}}\right] =g \\ \\
w(r,0)=w(r,\omega )=\dfrac{\partial w}{\partial \theta }(r,0)=\dfrac{%
\partial w}{\partial \theta }(r,\omega )=0 \\
w(\rho, \theta) = \dfrac{\partial^2 w}{\partial \theta^2 }(\rho,\theta)=\dfrac{\partial^2 w}{\partial \theta^2 }(\rho,\theta )=0. 
\end{array}\right. 
\end{equation}
\begin{Cor}
Let $g \in L^p_{\frac{1}{p}}(S_{\omega,\rho})$ with $p$ satisfying assumption \eqref{hyp inv sum 0}. Then, there exists $\rho_0 > 0$ such that problem \eqref{EquationEn RetTheta} has a unique solution $w$ satisfying
$$w \in W^{4,p}_{4 - \frac{1}{p}}(S_{\omega,\rho_0}).$$
\end{Cor}
This result is a direct consequence of \refT{Th principal}.
\begin{Rem}
Note that, we can also verify that we have in addition: for $i=0,1,2$ and $j = 0,1,2,3,4$ such that $0\leqslant i+j\leqslant 4$ 
$$ \frac{\partial^{i+j} w}{\partial r^i \partial \theta^j} \in L^p(S_{\omega,\rho_0}).$$ 
In fact, we have 
$$\dfrac{\partial^3 w}{\partial r^3}, \dfrac{\partial^4 w}{\partial r^3 \partial \theta} \in L^p_{3-\frac{1}{p}}(S_{\omega,\rho_0})\quad \text{and} \quad \dfrac{\partial^4 w}{\partial r^4} \in L^p_{4-\frac{1}{p}}(S_{\omega,\rho_0}),$$ 
see \refs{Sect Reg V2}; but, these two weighted spaces are not embedded into the space $L^p(S_{\omega,\rho_0})$.
\end{Rem}

\subsection{Vector formulation of problem \eqref{EquationEn RetTheta}}

Now, let us consider the vector variable $\Psi (r,\theta ):$ 
\begin{equation*}
\Psi =\left( 
\begin{array}{c}
w \\ \ecart
\displaystyle \left( r\dfrac{\partial }{\partial r}\right) ^{2}w
\end{array}%
\right) ,
\end{equation*}%
and the following matrix  
\begin{equation*}
\mathcal{A}=\left( 
\begin{array}{cc}
0 & 1 \\ \ecart
-\left( \dfrac{\partial ^{2}}{\partial \theta ^{2}}+1\right) ^{2} & -2\left( 
\dfrac{\partial ^{2}}{\partial \theta ^{2}}-1\right) 
\end{array}%
\right).
\end{equation*}
We have
\begin{eqnarray*}
&&\hspace*{-1cm}\left( r\dfrac{\partial }{\partial r}\right) ^{2}\Psi -\mathcal{A}\Psi  \\
&=&\left( 
\begin{array}{c}
\displaystyle \left( r\dfrac{\partial }{\partial r}\right) ^{2}w  \\ \ecart 
\displaystyle \left( r\dfrac{\partial }{\partial r}\right) ^{2}\left( r\dfrac{\partial }{
\partial r}\right) ^{2}w 
\end{array}
\right) +\left( 
\begin{array}{c}
\displaystyle -\left( r\dfrac{\partial }{\partial r}\right) ^{2} w  \\ \ecart
\displaystyle \left( \dfrac{\partial ^{2}}{\partial \theta ^{2}}+1\right) ^{2} w + 2 \left( \dfrac{\partial ^{2}}{\partial \theta ^{2}}-1\right)
\left( r\dfrac{\partial }{\partial r}\right) ^{2}w 
\end{array}%
\right)  \\ \ecart
&=&\left( 
\begin{array}{c}
0 \\ \ecart
\displaystyle \left( r\dfrac{\partial }{\partial r}\right) ^{2}\left( r\dfrac{\partial }{\partial r}\right) ^{2}w +2\left( \dfrac{\partial ^{2}}{\partial \theta ^{2}}-1\right) \left( r\dfrac{\partial }{\partial r}\right)^{2}w +\left( \dfrac{\partial ^{2}}{\partial \theta ^{2}}+1\right) ^{2}w 
\end{array}
\right) \\ \ecart
& = & \binom{0}{\Pi_1 +\Pi_2 } = \binom{0}{r^3 \Delta^2 u}.
\end{eqnarray*}

We set
$$\A_0 = \left( 
\begin{array}{cc}
0 & 0 \\ 
\dfrac{\partial ^{2}}{\partial \theta ^{2}}+1 & 1%
\end{array}%
\right) \quad \text{and}\quad \B_0 = \left( 
\begin{array}{cc}
0 & 0 \\ 
2\left( r\dfrac{\partial }{\partial r}\right)  & 0%
\end{array}\right).$$
We will precise the domain of all these operators in section \ref{sect espace de travail et domaine}. It is clear that the action of these operators are independent. Then
$$
\left(\mathcal{A}_0 + \mathcal{B}_0\right)\Psi = \left( 
\begin{array}{cc}
0 & 0 \\ \ecart
\displaystyle\dfrac{\partial ^{2}}{\partial \theta ^{2}}+1 & 1%
\end{array}%
\right) \left( 
\begin{array}{c}
w \\ \ecart
\displaystyle\left( r\dfrac{\partial }{\partial r}\right) ^{2} w
\end{array}%
\right) + \left( 
\begin{array}{cc}
0 & 0 \\ \ecart
2\left( r\dfrac{\partial }{\partial r}\right)  & 0%
\end{array}%
\right)  \left( 
\begin{array}{c}
w \\ \ecart
\displaystyle\left( r\dfrac{\partial }{\partial r}\right) ^{2}w 
\end{array}
\right),$$ 
hence
\begin{eqnarray*}
\left(\mathcal{A}_0 + \mathcal{B}_0\right) \Psi &=& \left( 
\begin{array}{c}
0 \\ \ecart
\displaystyle\dfrac{\partial ^{2}}{\partial \theta ^{2}} w + w +\left( r \dfrac{\partial }{\partial r}\right) ^{2}w 
\end{array}
\right) + \left( 
\begin{array}{c}
0 \\ \ecart
\displaystyle2\left( r\dfrac{\partial }{\partial r}\right) w 
\end{array}
\right) \\ \ecart
&=&\left( 
\begin{array}{c}
0 \\ \ecart
\displaystyle 2r\dfrac{\partial }{\partial r} w + \dfrac{\partial ^{2}}{%
\partial \theta ^{2}}w + w + \left( r\dfrac{\partial }{\partial r}\right) ^{2} w 
\end{array}
\right)  \\ \ecart
&=&\left( 
\begin{array}{c}
0 \\ 
r\Delta u
\end{array}
\right) = r \left( 
\begin{array}{c}
0 \\ 
\Delta u
\end{array}
\right) .
\end{eqnarray*}
The generalized diffusion equation becomes  
\begin{equation*}
\dfrac{1}{r^{3}}\left[ \left( r\dfrac{\partial }{\partial r}\right) ^{2}\Psi
-\mathcal{A}\Psi \right] -\dfrac{k}{r}\left(\mathcal{A}_0 + \mathcal{B}_0\right)\Psi =\left( 
\begin{array}{c}
0 \\ 
g
\end{array}
\right) .
\end{equation*}
Finally, we obtain the following complete equation 
\begin{equation}\label{eq complete}
\left[ \left( r\dfrac{\partial }{\partial r}\right) ^{2}\Psi -\mathcal{A}%
\Psi \right] -kr^{2}\mathcal{A}_{0}\Psi -kr^{2}\mathcal{B}_{0}\Psi =\left( 
\begin{array}{c}
0 \\ 
r^{3}g%
\end{array}%
\right) .
\end{equation}
Note that linear operators $\mathcal{A}$ and $\mathcal{A}_{0}$ act with respect to variable $\theta $ whereas operator $\mathcal{B}_{0}$ acts with respect to variable $r\dfrac{\partial }{\partial r}$.

\subsection{New formulation in a finite cone}

We apply the following variables and functions change  
\begin{equation*}
r = \rho e^{-t}, \quad \phi (t,\theta )= w(\rho
e^{-t},\theta ) \quad \text{and}\quad g(\rho e^{-t},\theta
)=G(t,\theta ),
\end{equation*}
then, it is easy to verify that 
\begin{equation*}
\left( r\dfrac{\partial }{\partial r}\right) w = - \frac{\partial \phi}{\partial t}  \quad \text{and} \quad \left( r\dfrac{\partial }{\partial r}\right) ^{2} w =\left[r\dfrac{\partial}{\partial r} + r^{2}\dfrac{\partial ^{2}}{\partial r^{2}}\right] w = \frac{\partial^2 \phi}{\partial t^2}.
\end{equation*}
We set $\Phi(t,\theta) = \Psi(\rho e^{-t}, \theta)$; then
\begin{equation*}
\Phi = \left( 
\begin{array}{c}
\phi \\ \ecart
\dis\frac{\partial^2 \phi}{\partial t^2}
\end{array}
\right).
\end{equation*}
Therefore, equation \eqref{eq complete} is now set on the strip $\Sigma = (0,+\infty) \times (0,\omega)$ in the following form
\begin{eqnarray*}
&&\hspace*{-2.5cm}\left[ \left( r\dfrac{\partial }{\partial r}\right) ^{2}\Psi -\mathcal{A}
\Psi \right] -kr^{2}\mathcal{A}_{0}\Psi -kr^{2}\mathcal{B}_{0}\left[ \Psi %
\right] \\
&=&\left[ \frac{\partial^2 \Phi}{\partial t^2} -\mathcal{A}\Phi \right] -k\rho ^{2}e^{-2t}\mathcal{A%
}_{0}\Phi +k\rho ^{2}e^{-2t}\mathcal{B}_{1}\left[ \Phi \right] =\left( 
\begin{array}{c}
0 \\ 
\rho ^{3}e^{-3t}G
\end{array}
\right),
\end{eqnarray*}
where
\begin{equation*}
\mathcal{B}_{1}=\left( 
\begin{array}{cc}
0 & 0 \\ 
\displaystyle -2\frac{\partial}{\partial t} & 0%
\end{array}%
\right) .
\end{equation*}
The boundary conditions, in problem \eqref{EquationEn RetTheta}, on $w$ become  
\begin{equation*}
\phi (r,0)=\phi (r,\omega )=\dfrac{\partial \phi }{\partial \theta }(r,0)=%
\dfrac{\partial \phi }{\partial \theta }(r,\omega )=0.
\end{equation*}
As usual, we will use the vector valued notation:
\begin{equation*}
\Phi (t)(\theta) := \Phi (t,\theta )=\left( 
\begin{array}{c}
\phi (t,\theta ) \\ \ecart 
\displaystyle \frac{\partial^2 \phi}{\partial t^2} (t,\theta )%
\end{array}%
\right) =\left( 
\begin{array}{c}
\phi (t,.) \\ \ecart
\displaystyle\frac{\partial^2 \phi}{\partial t^2} (t,.)%
\end{array}%
\right) (\theta ) :=\left( 
\begin{array}{c}
\phi (t) \\ \ecart
\displaystyle\frac{\partial^2 \phi}{\partial t^2} (t)%
\end{array}%
\right) (\theta ).
\end{equation*}
Hence, we deduce the following abstract vector valued equation 
\begin{equation*}
\left[ \Phi'' (t)-\mathcal{A}\Phi (t)\right] -k\rho
^{2}e^{-2t}\mathcal{A}_{0}\Phi (t)+k\rho ^{2}e^{-2t}\left[ \mathcal{B}%
_{1}\Phi \right] (t)=\left( 
\begin{array}{c}
0 \\ 
\rho ^{3}e^{-3t}G(t)%
\end{array}%
\right) ,
\end{equation*}
where
\begin{equation*}
\left[ \mathcal{B}_{1}\Phi \right] (t) = \left( 
\begin{array}{cc}
0 & 0 \\ \ecart
\displaystyle -2 \frac{\partial}{\partial t} & 0%
\end{array}%
\right) \left( 
\begin{array}{c}
\phi \\ \ecart
\displaystyle\frac{\partial^2 \phi}{\partial t^2}
\end{array}
\right) (t) = \left( 
\begin{array}{c}
0 \\ \ecart
\displaystyle -2 \frac{\partial \phi}{\partial t}(t)
\end{array}
\right).
\end{equation*}
Note that $\mathcal{A}$ and $\mathcal{A}_{0}$ act on $\Phi (t)$ with respect to $\theta$, while $\mathcal{B}_{1}$ acts on $\Phi $ with respect to $t$.

Now, we have to solve on $(0,+\infty)$ the following problem 
\begin{equation}
\left\{ 
\begin{array}{l}
\Phi ^{\prime \prime }(t)-\mathcal{A}\Phi (t)-k\rho ^{2}e^{-2t}\mathcal{A}_{0}\Phi (t)+ k\rho ^{2}e^{-2t} \left[ \mathcal{B}_{1}\Phi \right] (t)=\left( 
\begin{array}{c}
0 \\ 
\rho ^{3}e^{-3t}G(t)%
\end{array}%
\right) \\ 
\Phi (0)=0.
\end{array}
\right.  \label{ProbConeFini}
\end{equation}
\begin{Rem}\label{Rem Phi(+infini)=0}
Note that the boundary condition at $t=+\infty$, will be included in the vector valued space containing the solution $\Phi$.
\end{Rem}

\subsection{Sums of linear operators} \label{sect espace de travail et domaine}

In this section, we are going to write problem \eqref{ProbConeFini} as a sum of linear operators, firstly in the following Banach space 
\begin{equation*}
X=W_{0}^{2,p}(0,\omega) \times L^{p}(0,\omega),
\end{equation*}
see \eqref{ProbAbstrait} below and secondly, in $L^p(0,+\infty;X)$, see \eqref{Pb L1+L2} below.

Here $X$ is equipped, for instance, with the following norm 
\begin{equation*}
\left\Vert \left( 
\begin{array}{c}
\psi _{1} \\ 
\psi _{2}%
\end{array}%
\right) \right\Vert _{X}=\left\Vert \psi _{1}\right\Vert
_{W_{0}^{2,p}(0,\omega)}+\left\Vert \psi _{2}\right\Vert
_{L^{p}(0,\omega)},
\end{equation*}
where
\begin{equation*}
W_{0}^{2,p}(0,\omega) = \left\{ \varphi \in W^{2,p}(0,\omega) : \varphi (0) = \varphi (\omega ) = \varphi'(0) = \varphi'(\omega) = 0\right\} .
\end{equation*}
Then, we define the linear operator $\mathcal{A}$ by
\begin{equation*}
\left\{ 
\begin{array}{lll}
D(\mathcal{A}) & = & \dis \left[ W^{4,p}(0,\omega)\cap W_{0}^{2,p}(0,\omega)\right] \times W_{0}^{2,p}(0,\omega)\subset X \\ \ecart
\mathcal{A}\left( 
\begin{array}{c}
\psi _{1} \\ 
\psi _{2}
\end{array}
\right) &=& \left( 
\begin{array}{c}
\psi _{2} \\ 
-\left( \dfrac{\partial ^{2}}{\partial \theta ^{2}}+1\right) ^{2}\psi
_{1}-2\left( \dfrac{\partial ^{2}}{\partial \theta ^{2}}-1\right) \psi _{2}%
\end{array}%
\right), \quad \left( 
\begin{array}{c}
\psi _{1} \\ 
\psi _{2}%
\end{array}%
\right) \in D(\A).
\end{array}
\right.
\end{equation*}
In the same way, we define operator $\mathcal{A}_{0}$ by 
\begin{equation*}
\left\{ 
\begin{array}{lll}
D(\mathcal{A}_{0}) & = & W_{0}^{2,p}(0,\omega) \times L^{p}(0,\omega) = X \\ \ecart
\mathcal{A}_{0}\left( 
\begin{array}{c}
\psi _{1} \\ 
\psi _{2}
\end{array}
\right) & = & \dis \left( 
\begin{array}{c}
0 \\ 
\left( \dfrac{\partial ^{2}}{\partial \theta ^{2}}+1\right) \psi _{1}+\psi
_{2}%
\end{array}%
\right), \quad \left( 
\begin{array}{c}
\psi _{1} \\ 
\psi _{2}%
\end{array}%
\right) \in D(\A_0).%
\end{array}%
\right.
\end{equation*}
It is clear that $D(\mathcal{A})\subset D(\mathcal{A}_{0})$. Note that operator $\mathcal{A}_{0}$ is continuous from $X$ into itself since
\begin{eqnarray*}
\left\Vert \mathcal{A}_{0}\left( 
\begin{array}{c}
\psi _{1} \\ 
\psi _{2}%
\end{array}%
\right) \right\Vert _{X} &=&\left\Vert \left( 
\begin{array}{c}
0 \\ 
\left( \dfrac{\partial ^{2}}{\partial \theta ^{2}}+1\right) \psi _{1}+\psi
_{2}
\end{array}
\right) \right\Vert _{X}=\left\Vert \left( \dfrac{\partial ^{2}}{\partial
\theta ^{2}}+1\right) \psi _{1}+\psi _{2}\right\Vert _{L^{p}(0,\omega)} \\
&\leqslant &\left\Vert \psi _{1}\right\Vert _{W_{0}^{2,p}(0,\omega)} + \left\Vert \psi _{2}\right\Vert _{L^{p}(0,\omega)}=\left\Vert
\left( 
\begin{array}{c}
\psi _{1} \\ 
\psi _{2}%
\end{array}%
\right) \right\Vert _{X}.
\end{eqnarray*}
Equation \eqref{ProbConeFini} is set in the Banach space $X$.

Recall that the second member in problem \eqref{Pb cone infini} satisfies
$$f \in L^p\left(S_{\omega,\rho}\right), \quad \text{for } p \in (1,+\infty).$$
Set
\begin{equation*}
t\mapsto e^{-3t}G(t)(.)=e^{-3t}G(t,.)=H(t,.)=H(t)(.).
\end{equation*}
Therefore, we have 
\begin{eqnarray*}
\int_{S_{\omega,\rho }}\left\vert f(x,y)\right\vert ^{p}dxdy &=&\int_{S_{\omega,\rho
}}\left\vert g(r,\theta )\right\vert ^{p}rdrd\theta =\rho ^{2}\int_{%
\Sigma }\left\vert G(t,\theta )\right\vert ^{p}e^{-2t}dtd\theta \\
&=&\rho ^{2}\int_{\Sigma }\left\vert e^{\left( 3-\frac{2}{p}\right) t}H(t,\theta )\right\vert ^{p}dtd\theta \\
&=&\rho ^{2}\int_{0}^{+\infty }\left\vert e^{\left( 3-\frac{2}{p}\right)
t}\right\vert ^{p}\left[ \left( \int_{0}^{\omega }\left\vert H(t)(\theta
)\right\vert ^{p}d\theta \right) ^{1/p}\right] ^{p}dt \\
&=&\rho ^{2}\int_{0}^{+\infty } e^{\left( 3p-2\right)
t}\left\Vert H(t)\right\Vert _{L^{p}(0,\omega)}^{p}dt.
\end{eqnarray*}
It follows that the second member $H$ is in the weighted Sobolev space as recalled in \eqref{Def Wmp}
\begin{equation*}
\left\{H : t\mapsto e^{\left( 3-\frac{2}{p}\right) t}H\in L^{p}(\Sigma)\right\} = L_{\nu}^{p}(0,+\infty ;L^{p}(0,\omega)),
\end{equation*}
where 
$$\nu = 3 - \frac{2}{p}\in (1,3),$$
which is exactly the Sobolev exponent of the space $W^{3,p}(\Sigma)$ in dimension 2.

Then, since it would not be easy to work in weighted Sobolev spaces, we will use the following new vector valued function :
\begin{equation}\label{V = Phi}
V(t)=e^{\nu t}\Phi (t) = \left( 
\begin{array}{c}
e^{\nu t}\phi (t) \\ 
e^{\nu t} \phi''(t)
\end{array}%
\right) =\left( 
\begin{array}{c}
V_{1}(t) \\ 
V_{2}(t)
\end{array}\right).
\end{equation}
Since we have
\begin{equation*}
\Phi (t)=\left( 
\begin{array}{c}
e^{-\nu t}V_{1}(t) \\ 
e^{-\nu t}V_{2}(t)
\end{array}
\right),
\end{equation*}
we deduce that
\begin{equation*}
\Phi'(t) = -\nu e^{-\nu t}V(t)+e^{-\nu t}V'(t) = e^{-\nu t}\left(
\partial_{t}-\nu I\right) V(t) ,
\end{equation*}
and
\begin{equation*}
\Phi''(t) = \nu ^{2}e^{-\nu t}V(t) - 2\nu e^{-\nu t} V'(t) + e^{-\nu t} V''(t) = e^{-\nu t}\left( \partial_{t}-\nu I\right)^{2}V(t).
\end{equation*}
Moreover, we obtain
\begin{eqnarray*}
\left[ \mathcal{B}_{1}\Phi \right] (t) &=&\left( 
\begin{array}{cc}
0 & 0 \\ 
-2\partial_{t} & 0
\end{array}\right) \left( 
\begin{array}{c}
e^{-\nu t}V_{1}(t) \\ 
e^{-\nu t}V_{2}(t)%
\end{array}%
\right) \\
&=&e^{-\nu t}\left( 
\begin{array}{cc}
0 & 0 \\ 
-2(\partial_{t}-\nu I) & 0%
\end{array}%
\right) \left( 
\begin{array}{c}
V_{1}(t) \\ 
V_{2}(t)%
\end{array}%
\right) \\
&=&e^{-\nu t}\left[ \mathcal{B}_{2,\nu}V\right] (t),
\end{eqnarray*}%
where
\begin{equation*}
\mathcal{B}_{2,\nu}=\left( 
\begin{array}{cc}
0 & 0 \\ 
-2(\partial_{t}-\nu I) & 0
\end{array}
\right) .
\end{equation*}
Hence, problem \eqref{ProbConeFini} becomes
\begin{equation*}
\left\{ 
\begin{array}{l}
e^{-\nu t}\left( \partial_{t}-\nu I\right)^{2}V(t)-e^{-\nu t}\mathcal{A} V(t)-k\rho ^{2}e^{-\nu t}e^{-2t}\mathcal{A}_{0}V(t) \\ 
-k\rho ^{2}e^{-\nu t}e^{-2t}\left[ (\mathcal{B}_{2,\nu}V)\right] (t)=\left( 
\begin{array}{c}
0 \\ 
\rho ^{3}H(t)
\end{array}
\right) \\ 
V(0)=0,
\end{array}
\right.
\end{equation*}
then
\begin{equation}\label{ProbAbstrait}
\left\{ \hspace*{-0.16cm}
\begin{array}{l}
\left( \partial_{t}-\nu I\right) ^{2}V(t)-\mathcal{A}V(t)-k\rho ^{2}e^{-2t}
\mathcal{A}_{0}V(t) -k\rho ^{2}e^{-2t}\left[ (\mathcal{B}_{2,\nu}V)\right] (t)\hspace*{-0.05cm}=\hspace*{-0.05cm}\left( 
\begin{array}{c}
0 \\ 
\rho ^{3}e^{\nu t}H(t)
\end{array}\right) \\ 
V(0)=0.
\end{array}
\right.  
\end{equation}
We have 
\begin{equation*}
t\mapsto \rho ^{3}e^{\nu t}H(t)\in L^{p}(\Sigma) = L^{p}(0,+\infty ;L^{p}(0,\omega)),
\end{equation*}
and
\begin{equation*}
t\mapsto \left( 
\begin{array}{c}
0 \\ 
\rho ^{3}e^{\nu t}H(t)
\end{array}
\right) \in L^{p}\left( 0,+\infty ;W_{0}^{2,p}(0,\omega)\times L^{p}(0,\omega)\right) = L^{p}(0,+\infty ;X).
\end{equation*}
Finally, let us introduce the following abstract linear operators: 
\begin{equation*}
\left\{ \begin{array}{cll}
D(\mathcal{L}_{1,\nu}) & = & \dis \left\{ V\in W^{2,p}(0,+\infty ;X): V(0) = V(+\infty) = 0\right\} \\ \ecart
\left[ \mathcal{L}_{1,\nu}(V)\right] (t) & = & \dis \left( \partial_{t}-\nu I\right)^{2}V(t)=V''(t)-2\nu V'(t)+\nu^{2}V(t),
\end{array}\right.
\end{equation*}
with $\dis\nu = 3 - \frac{2}{p} \in (1,3)$, 
\begin{equation*}
\left\{ 
\begin{array}{lll}
D(\mathcal{L}_{2}) & = &\dis \left\{ V\in L^{p}(0,+\infty ;X): \text{for }a.e.~t\in
(0,+\infty ),~ V(t)\in D(\mathcal{A})\right\} \\ \ecart
\left[ \mathcal{L}_{2}(V)\right] (t) & = & -\mathcal{A}V(t),
\end{array}
\right.
\end{equation*}
\begin{equation*}
\left\{ 
\begin{array}{lll}
D(\mathcal{P}_{1}) & = &\dis \left\{ V\in L^{p}(0,+\infty ;X): \text{for }a.e.~t\in
(0,+\infty ),~ V(t)\in D(\mathcal{A}_0)\right\} \\ \ecart
\left[ \mathcal{P}_{1}(V)\right] (t) & = & -e^{-2t} \mathcal{A}_{0}V(t),
\end{array}
\right.
\end{equation*}
and
\begin{equation*}
\left\{ 
\begin{array}{lll}
D(\mathcal{P}_{2,\nu}) & = &\dis W^{1,p}(0,+\infty ;X) \\ \ecart
\left[ \mathcal{P}_{2,\nu}(V)\right] (t) & = & -e^{-2t}\left( \mathcal{B}_{2,\nu}V\right) (t).
\end{array}\right.
\end{equation*}
Then, problem \eqref{ProbAbstrait} can be written as the following abstract equation
\begin{equation}\label{Pb L1+L2}
\left(\L_{1,\nu} + \L_2\right) V + k \rho^2 \left(\P_1 + \P_{2,\nu} \right)V = \F_\nu,
\end{equation}
set in $L^p(0,+\infty;X)$, with $p \in (1,+\infty)$, where, for almost every $t \in (0,+\infty)$
$$\F_\nu(t) =  \left( 
\begin{array}{c}
0 \\ 
\rho ^{3}e^{\nu t}H(t)
\end{array}\right).$$

\section{Proof of \refT{Th principal}} \label{Sect proof of main Th}

\subsection{Resolution of equation \eqref{Pb L1+L2}}

Equation \eqref{Pb L1+L2} will be completely studied in the second part of this work by using the sum theory of linear operators, where the main result described by Theorem 1.1 in Labbas, Maingot and Thorel \cite{Cone P2} states that there exists $\rho_0 >0$ such that for all $\rho \in(0,\rho_0]$, there exists a unique solution $V \in D(\L_{1,\nu}+\L_2)$ to equation \eqref{Pb L1+L2} that is
$$V \in W^{2,p}(0,+\infty;X) \cap L^p(0,+\infty;D(\A)).$$
Thus, we know that there exists a continuous extension from $W^{2,p}(0,+\infty;X)$ into $W^{2,p}(\RR;X)$ and also from $L^p(0,+\infty;D(\A))$ into $L^p(\RR;D(\A))$; it suffices, for instance to use the well-known Babich techniques.

Set $\widetilde{V}$, the extension of $V$; it is then written as 
$$\widetilde{V} = \left(\begin{array}{c}
\widetilde{V_1} \\
\widetilde{V_2}
\end{array}\right).$$
So  
$$\widetilde{V_1} \in W^{2,p}\left(\RR; W^{2,p}_0(0,\omega)\right) \cap L^p\left(\RR; W^{4,p}(0,\omega)\cap  W^{2,p}_0(0,\omega)\right),$$
and
$$\widetilde{V_2} \in W^{2,p}\left(\RR; L^p(0,\omega)\right) \cap L^p\left(\RR; W^{2,p}_0(0,\omega)\right).$$
From this two properties, we obtain
$$\widetilde{V_1}, \widetilde{V_2} \in W^{2,p}\left(\RR \times (0,\omega)\right),$$
and 
$$\widetilde{V_1} \in W^{2,p}\left(\RR; W^{2,p}_0(0,\omega)\right) \cap L^p\left(\RR; W^{4,p}(0,\omega)\right),$$
by using the following lemma.  
\begin{Lem}
Let $I=\;]a,b[$, $a<b$ be an open bounded subset of $\RR$. Then
$$W^{2,p}(\RR; L^p(I)) \cap L^p(\RR; W^{2,p}(I)) \subset W^{2,p}(\RR \times I).$$
\end{Lem}
\begin{proof} 
Let $\chi \in W^{m,p}(a,b)$. Then $\chi \in C^{m-1}(\left[ a,b\right])$. Consider the Taylor polynomial  
$$
T_{m-1}(x) =\sum_{k=0}^{m-1}\frac{1}{k!}( x-b)^{k}\chi^{(k)}(b), \quad \text{for }x>b.
$$
We define the following bump function $\varphi \in C^{\infty }([b,+\infty))$ such that
$$\begin{array}{lll}
\varphi(x) &=&\left \{ 
\begin{array}{ll}
1 & \text{ for }b\leq x\leq b+1 \\
0 & \text{ for }x\geq b+2.
\end{array}
\right.\end{array}$$
Then $x \longmapsto T_{m-1}( x) \varphi (x)$ belongs to $W^{m,p}(b,+\infty)$. Now, we define an extension function on $(a,+\infty)$ by
$$\tilde{\chi}(x) = \left\{ 
\begin{array}{ll}
\chi(x) & \text{ for }a<x<b \\ 
T_{m-1}(x) \varphi(x) &  \text{ for }x>b.
\end{array}
\right.$$
Clearly, $\tilde{\chi}$ belongs to $W^{m,p}(a,+\infty)$. Moreover, all the derivatives on the left and on the right, up to order $m-1$, coincide. In the same way, we build an extension to the left of point a. Consequently, there exists an extension operator $P$ which maps continuously 
$$L^p(I) \quad \text{into} \quad L^p(\RR) \quad \text{and} \quad W^{2,p}(I)\quad \text{into} \quad W^{2,p}(\RR).$$
Therefore, for
$$\chi \in W^{2,p}(\RR; L^p(I)) \cap L^p(\RR; W^{2,p}(I)),$$
we have
$$P\chi \in W^{2,p}(\RR; L^p(\RR)) \cap L^p(\RR; W^{2,p}(\RR)).$$
This last space coincides with 
$$W^{2,p}(\RR^{2}),$$
by Mihlin's theorem (see Mihlin \cite{mihlin}).

Consequently, $\chi$, the restriction of $P\chi$ to $\RR \times I$, belongs to  $W^{2,p}(\RR \times I)$.
\end{proof}
We take this opportunity to indicate a similar result in the case of a bounded open set of $\RR^n$ with $n>1$.
\begin{Lem}
Let $U$ be an open bounded subset of $\RR^n$, $n > 1$, with a Lipschitz boundary. Then
$$W^{2,p}(\RR; L^p(U)) \cap L^p(\RR; W^{2,p}(U)) \subset W^{2,p}(\RR \times U).$$
\end{Lem}
\begin{proof}
We know that there exists an extension operator $P$ which maps continuously 
$$L^p(U) \quad \text{into} \quad L^p(\RR^n) \quad \text{and} \quad W^{2,p}(U)\quad \text{into} \quad W^{2,p}(\RR^n);$$
then for 
$$\chi \in W^{2,p}(\RR; L^p(U)) \cap L^p(\RR; W^{2,p}(U)),$$
we have
$$P\chi \in W^{2,p}(\RR; L^p(\RR^n)) \cap L^p(\RR; W^{2,p}(\RR^n)).$$
This last space coincide with 
$$W^{2,p}(\RR^{n+1}),$$
by Mihlin's theorem (see Mihlin \cite{mihlin}).

Consequently, $\chi$, the restriction of $P\chi$ to $\RR \times U$, belongs to  $W^{2,p}(\RR \times U)$.
\end{proof}

Therefore, we deduce that $V_1$ and $V_2$ have the same regularities on $(0,\rho) \times (0,\omega)$, with $\rho \in (0,\rho_0]$.
 
\subsection{Regularities of $v(r,\theta)$ and $w(r,\theta)=v(r,\theta)/r$}

In the sequel, all the regularities of $w$ are deduced easily from those of $v$.

Recall that, from \eqref{V = Phi}, we have
$$V(t)=e^{\nu t}\Phi (t) \quad \text{and} \quad \Phi (t) = e^{-\nu t}V(t).$$
Moreover, since $r = \rho e^{-t}$, we obtain
$$V_1 (t,\theta) = \frac{v(\rho e^{-t},\theta)}{\left(\rho e^{-t}\right)^{\nu+1}},$$
where $\dis\nu = 3 - \frac{2}{p} \in (1,3)$ and
$$V_2 (t,\theta) = \left(r \frac{\partial}{\partial r}\right)^2 \left(\frac{v}{r}\right)(\rho e^{-t},\theta) = (\rho e^{-t})\frac{\partial^2 v}{\partial r^2}(\rho e^{-t},\theta) - \frac{\partial v}{\partial r}(\rho e^{-t},\theta) + \frac{v(\rho e^{-t},\theta)}{\rho e^{-t}}.$$

\subsubsection{Regularity of $V_1$}\label{SubSect First situation}

Here, we explicit the fact that 
\begin{equation}\label{V1 in W2p}
V_1 \in W^{2,p}\left((0,+\infty)\times(0,\omega)\right).
\end{equation}
We have
$$\int_0^{+\infty} \int_0^\omega \left|V_1(t,\theta)\right|^p d\theta~dt = \int_0^{+\infty} \int_0^\omega \left|\frac{v(\rho e^{-t},\theta)}{\rho^{\nu+1} e^{-t(\nu+1)}}\right|^p d\theta~dt < + \infty.$$
Setting $r = \rho e^{-t}$, we obtain
$$\begin{array}{lll}
\dis\int_0^{+\infty} \int_0^\omega \left|\frac{v(\rho e^{-t},\theta)}{\left(\rho e^{-t}\right)^{\nu+1}}\right|^p d\theta~dt & = & \dis\int_0^\rho \int_0^\omega \left|\frac{v\left(r,\theta\right)}{r^{\nu + 1 + \frac{1}{p}}}\right|^p d\theta~dr \\ \ecart

& = & \dis \int_0^\rho \int_0^\omega r^{-4p + 1}\left|v\left(r,\theta\right)\right|^p d\theta~dr.
\end{array}$$
Then, we have
\begin{equation}\label{v in Lp gamma1}
v \in L^p_{\gamma_0}(S_{\omega,\rho}) \quad \text{and} \quad w=\frac{v}{r} \in L^p_{\gamma_0 + 1}(S_{\omega,\rho}),
\end{equation}
where $\gamma_0 = -4 + \dfrac{1}{p}$. Moreover
$$\begin{array}{lll}
\dis\frac{\partial V_1}{\partial t}(t,\theta) &=& \dis \frac{1}{\rho^{\nu + 1}}\frac{\partial}{\partial t}\left(v(\rho e^{-t},\theta) e^{(\nu+1)t}\right) \\ \ecart
& = & \dis\frac{1}{\rho^{\nu + 1}}\left((\nu+1)e^{(\nu+1)t} v(\rho e^{-t},\theta) - \rho e^{\nu t} \frac{\partial v}{\partial r}(\rho e^{-t},\theta)  \right),
\end{array}$$
hence
$$ \frac{1}{\left(\rho e^{- t}\right)^{\nu}} \frac{\partial v}{\partial r}(\rho e^{-t},\theta) = \frac{(\nu+1)}{\left(\rho e^{-t}\right)^{\nu+1}} v(\rho e^{-t},\theta) - \frac{\partial V_1(t,\theta)}{\partial t}.$$
Thus, in virtue of \eqref{V1 in W2p} and \eqref{v in Lp gamma1}, it follows that
$$(t,\theta) \longmapsto \frac{1}{\left(\rho e^{- t}\right)^{\nu}} \frac{\partial v}{\partial r}(\rho e^{-t},\theta) \in L^p((0,+\infty)\times (0,\omega)),$$
and
$$\begin{array}{lll}
\dis\int_0^{+\infty} \int_0^\omega \left|\frac{1}{\left(\rho e^{- t}\right)^{\nu}} \frac{\partial v}{\partial r}(\rho e^{-t},\theta)\right|^p d\theta~dt & = & \dis \int_0^{\rho} \int_0^\omega \left|\frac{1}{r^{\nu + \frac{1}{p}}} \frac{\partial v}{\partial r}(r,\theta)\right|^p d\theta~dr \\ \ecart

& = & \dis \int_0^{\rho} \int_0^\omega r^{-3p + 1} \left| \frac{\partial v}{\partial r}(r,\theta)\right|^p d\theta~dr.
\end{array}$$
So, we obtain
\begin{equation}\label{v' in Lp gamma2}
\frac{\partial v}{\partial r} \in L^p_{\gamma_1}(S_{\omega,\rho}) \quad \text{and} \quad \frac{\partial w}{\partial r} \in L^p_{\gamma_1 + 1}(S_{\omega,\rho}),
\end{equation}
where $\gamma_1 = -3 + \dfrac{1}{p}$. Furthermore, we have
$$\begin{array}{lll}
\dis\frac{\partial^2 V_1}{\partial t^2}(t,\theta) &=& \dis \frac{\nu+1}{\rho^{\nu + 1}}\frac{\partial}{\partial t}\left(v(\rho e^{-t},\theta) e^{(\nu+1)t}\right) - \frac{\nu}{\rho^{\nu}} e^{\nu t} \frac{\partial v}{\partial r}(\rho e^{-t},\theta) + \frac{\rho}{\rho^{\nu}} e^{-t}e^{\nu t} \frac{\partial^2 v}{\partial r^2}(\rho e^{-t},\theta)  \\ \\
& = & \dis\frac{\nu+1}{\rho^{\nu + 1}}\left((\nu+1)e^{(\nu+1)t} v(\rho e^{-t},\theta) - \rho e^{\nu t} \frac{\partial v}{\partial r}(\rho e^{-t},\theta)  \right) - \frac{\nu}{\rho^{\nu}} e^{\nu t} \frac{\partial v}{\partial r}(\rho e^{-t},\theta) \\ \ecart
&&\dis + \frac{1}{\rho^{\nu -1}} e^{(\nu - 1)t} \frac{\partial^2 v}{\partial r^2}(\rho e^{-t},\theta) \\ \\

& = & \dis \frac{(\nu+1)^2}{\left(\rho e^{-t}\right)^{\nu + 1}}  v(\rho e^{-t},\theta) - \frac{2\nu + 1}{\left(\rho e^{- t}\right)^{\nu}}  \frac{\partial v}{\partial r}(\rho e^{-t},\theta) + \frac{1}{\left(\rho e^{-t}\right)^{\nu-1}}  \frac{\partial^2 v}{\partial r^2}(\rho e^{-t},\theta),
\end{array}$$
hence
$$\frac{1}{\left(\rho e^{-t}\right)^{\nu-1}}  \frac{\partial^2 v}{\partial r^2}(\rho e^{-t},\theta) = \frac{\partial^2 V_1}{\partial t^2}(t,\theta) + \frac{2\nu + 1}{\left(\rho e^{- t}\right)^{\nu}}  \frac{\partial v}{\partial r}(\rho e^{-t},\theta) - \frac{(\nu+1)^2}{\left(\rho e^{-t}\right)^{\nu + 1}}  v(\rho e^{-t},\theta).$$
Thus, in virtue of \eqref{V1 in W2p}, \eqref{v in Lp gamma1} and \eqref{v' in Lp gamma2}, it follows that
$$(t,\theta) \longmapsto \frac{1}{\left(\rho e^{- t}\right)^{\nu - 1}} \frac{\partial^2 v}{\partial r^2}(\rho e^{-t},\theta) \in L^p((0,+\infty)\times (0,\omega)),$$
and
$$\begin{array}{lll}
\dis\int_0^{+\infty} \int_0^\omega \left|\frac{1}{\left(\rho e^{- t}\right)^{\nu -1}} \frac{\partial^2 v}{\partial r^2}(\rho e^{-t},\theta)\right|^p d\theta~dt & = & \dis \int_0^{\rho} \int_0^\omega \left|\frac{1}{r^{\nu - 1 + \frac{1}{p}}} \frac{\partial^2 v}{\partial r^2}(r,\theta)\right|^p d\theta~dr \\ \ecart

& = & \dis \int_0^{\rho} \int_0^\omega r^{-2p + 1} \left| \frac{\partial^2 v}{\partial r^2}(r,\theta)\right|^p d\theta~dr.
\end{array}$$
So, we obtain
\begin{equation*}
\frac{\partial^2 v}{\partial r^2} \in L^p_{\gamma_2}(S_{\omega,\rho}) \quad \text{and} \quad \frac{\partial^2 w}{\partial r^2} \in L^p_{\gamma_2 + 1}(S_{\omega,\rho}),
\end{equation*}
where $\gamma_2 = -2 + \dfrac{1}{p}$.

On the other hand
$$\frac{\partial V_1}{\partial \theta} (\rho e^{-t}, \theta) = \frac{1}{\left(\rho e^{-t} \right)^{\nu + 1}}\frac{\partial v}{\partial \theta} (\rho e^{-t}, \theta).$$
Thus, it follows that
$$\begin{array}{lll}
\dis\int_0^{+\infty} \int_0^\omega \left| \frac{\partial V_1}{\partial \theta}(\rho e^{-t},\theta)\right|^p d\theta~dt & = & \dis \int_0^{+\infty} \int_0^\omega \frac{1}{\left(\rho e^{- t}\right)^{p(\nu + 1)}} \left|\frac{\partial v}{\partial \theta}(\rho e^{-t},\theta)\right|^p d\theta~dt \\ \ecart

& = & \dis \int_0^{\rho} \int_0^\omega r^{-4p+1} \left|\frac{\partial v}{\partial \theta}(r,\theta)\right|^p d\theta~dr.
\end{array}$$
So, we obtain
\begin{equation*}
\frac{\partial v}{\partial \theta} \in L^p_{\gamma_0}(S_{\omega,\rho}) \quad \text{and} \quad \frac{\partial w}{\partial \theta} \in L^p_{\gamma_0 + 1}(S_{\omega,\rho}).
\end{equation*}
In the same way, we deduce that
$$\frac{\partial^2 v}{\partial \theta^2} \in L^p_{\gamma_0}(S_{\omega,\rho}) \quad \text{and} \quad \frac{\partial^2 v}{\partial r \partial \theta} \in L^p_{\gamma_1}(S_{\omega,\rho}),$$
hence
$$\frac{\partial^2 w}{\partial \theta^2} \in L^p_{\gamma_0+1}(S_{\omega,\rho}) \quad \text{and} \quad \frac{\partial^2 w}{\partial r \partial \theta} \in L^p_{\gamma_1+1}(S_{\omega,\rho}).$$
Now, we explicit the fact that  
$$\frac{\partial^3 V_1}{\partial t \partial\theta^2}, \frac{\partial^3 V_1}{\partial t^2\partial\theta}, \frac{\partial^4 V_1}{\partial t^2\partial\theta^2} \in L^p \left((0,+\infty)\times (0,\omega)\right).$$ 
We have 
$$\frac{\partial^3 V_1}{\partial t^2 \partial \theta}(t,\theta) = \frac{(\nu+1)^2}{\left(\rho e^{-t}\right)^{\nu + 1}}  \frac{\partial v}{\partial \theta}(\rho e^{-t},\theta) - \frac{2\nu + 1}{\left(\rho e^{- t}\right)^{\nu}}  \frac{\partial^2 v}{\partial r \partial\theta}(\rho e^{-t},\theta) + \frac{1}{\left(\rho e^{-t}\right)^{\nu-1}}  \frac{\partial^3 v}{\partial r^2 \partial \theta}(\rho e^{-t},\theta),$$
hence
$$\frac{1}{\left(\rho e^{-t}\right)^{\nu-1}}  \frac{\partial^3 v}{\partial r^2 \partial \theta}(\rho e^{-t},\theta) = \frac{\partial^3 V_1}{\partial t^2 \partial \theta}(t,\theta) - \frac{(\nu+1)^2}{\left(\rho e^{-t}\right)^{\nu + 1}}  \frac{\partial v}{\partial \theta}(\rho e^{-t},\theta) + \frac{2\nu + 1}{\left(\rho e^{- t}\right)^{\nu}}  \frac{\partial^2 v}{\partial r \partial\theta}(\rho e^{-t},\theta).$$
Thus, we obtain
$$\frac{\partial^3 v}{\partial r^2 \partial \theta} \in L^p_{\gamma_2}(S_{\omega,\rho})\quad \text{and} \quad \frac{\partial^3 w}{\partial r^2 \partial \theta} \in L^p_{\gamma_2 + 1}(S_{\omega,\rho}),$$
and in the same way, we also have 
$$\frac{\partial^3 v}{\partial r \partial\theta^2} \in L^p_{\gamma_1}(S_{\omega,\rho}) \quad \text{and} \quad \frac{\partial^4 v}{\partial r^2 \partial \theta^2} \in L^p_{\gamma_2}(S_{\omega,\rho}),$$
hence
$$\frac{\partial^3 w}{\partial r \partial\theta^2} \in L^p_{\gamma_1+1}(S_{\omega,\rho}) \quad \text{and} \quad \frac{\partial^4 w}{\partial r^2 \partial \theta^2} \in L^p_{\gamma_2+1}(S_{\omega,\rho}).$$
Now, we explicit the fact that
$$V_1 \in L^p\left((0,+\infty);W^{4,p}(0,\omega)\right),$$
that is, for all $i = 1,2,3,4$
$$\int_0^{+\infty} \left\|\frac{\partial^i V_1}{\partial \theta^i}(t,\theta)\right\|^p_{L^p(0,\omega)} dt < +\infty.$$
Then, we have
$$\begin{array}{lll}
\dis\int_0^{+\infty} \int_0^\omega \left|\frac{\partial^i V_1}{\partial \theta^i}(t,\theta)\right|^p d\theta~dt & = &\dis\int_0^{+\infty} \int_0^\omega \left|\frac{1}{\left(\rho e^{-t}\right)^{\nu+1}}\frac{\partial^i v}{\partial\theta^i}(\rho e^{-t},\theta)\right|^p d\theta~dt \\ \ecart

& = & \dis \int_0^\rho \int_0^\omega r^{-4p + 1}\left|\frac{\partial^i v}{\partial\theta^i}\left(r,\theta\right)\right|^p d\theta~dr,
\end{array}$$
which gives
$$\frac{\partial^i v}{\partial\theta^i} \in L^p_{\gamma_0}(S_{\omega,\rho}) \quad \text{and} \quad \frac{\partial^i w}{\partial \theta^i} \in L^p_{\gamma_0 + 1}(S_{\omega,\rho}), \quad \text{for }i = 1,2,3,4.$$

\subsubsection{Regularity of $V_2$}\label{Sect Reg V2}

In the same way, we explicit the fact that
$$V_2 \in W^{2,p}\left((0,+\infty)\times (0,\omega)\right),$$
where 
$$V_2 (t,\theta) = \left(r \frac{\partial}{\partial r}\right)^2 \left(\frac{v}{r}\right)(\rho e^{-t},\theta) = \rho e^{-t}\frac{\partial^2 v}{\partial r^2}(\rho e^{-t},\theta) - \frac{\partial v}{\partial r}(\rho e^{-t},\theta) + \frac{v(\rho e^{-t},\theta)}{\rho e^{-t}}.$$
It is clear, from Subsection~\ref{SubSect First situation}, that
$$r\frac{\partial^2 v}{\partial r^2},~\frac{\partial v}{\partial r},~\frac{v}{r} \in L^p_{\gamma_1}(S_{\omega,\rho}).$$
Moreover, we have
$$\frac{\partial V_2}{\partial t}(t,\theta) = -\left(\rho e^{-t}\right)^2 \frac{\partial^3 v}{\partial r^3} (\rho e^{-t}, \theta) - \frac{\partial v}{\partial r} (\rho e^{-t}, \theta) + \frac{v(\rho e^{-t},\theta)}{\rho e^{-t}},$$
hence
$$\left(\rho e^{-t}\right)^2 \frac{\partial^3 v}{\partial r^3} (\rho e^{-t}, \theta) = - \frac{\partial v}{\partial r} (\rho e^{-t}, \theta) + \frac{v(\rho e^{-t},\theta)}{\rho e^{-t}} - \frac{\partial V_2}{\partial t}(t,\theta).$$
Then
$$\begin{array}{lll}
\dis\int_0^{+\infty} \int_0^\omega \left| \frac{\partial v}{\partial r} (\rho e^{-t}, \theta) \right|^p d\theta~dt & = &\dis\int_0^{\rho} \int_0^\omega \frac{r^{3p - 1}}{r}r^{-3p+1}\left|\frac{\partial v}{\partial r}(r,\theta)\right|^p d\theta~dr \\ \ecart

& \leqslant & \dis \rho^{3p - 2} \int_0^\rho \int_0^\omega  r^{-3p+1}\left|\frac{\partial v}{\partial r}(r,\theta)\right|^p d\theta~dr < + \infty.
\end{array}$$
It follows that
$$\int_0^{+\infty} \int_0^\omega \left|\left(\rho e^{-t}\right)^2 \frac{\partial^3 v}{\partial r^3}(\rho e^{-t},\theta)\right|^p d\theta~dt = \int_0^{\rho} \int_0^\omega r^{2p - 1}\left|\frac{\partial^3 v}{\partial r^3}(r,\theta)\right|^p d\theta~dr < + \infty,$$
which means that 
\begin{equation*}
\frac{\partial^3 v}{\partial r^3} \in L^p_{\gamma_3}(S_{\omega,\rho}),\quad \text{and} \quad \frac{\partial^3 w}{\partial r^3} \in L^p_{\gamma_3 + 1}(S_{\omega,\rho}),
\end{equation*}
where $\gamma_3 = 2 - \dfrac{1}{p}$.

In the same way, we have
\begin{equation*}
\frac{\partial^4 v}{\partial r \partial \theta^3} \in L^p_{\gamma_1}(S_{\omega,\rho}) \quad \text{and} \quad \frac{\partial^4 v}{\partial r^3 \partial \theta} \in L^p_{\gamma_3}(S_{\omega,\rho}),
\end{equation*}
hence
$$\frac{\partial^4 w}{\partial r \partial \theta^3} \in L^p_{\gamma_1+1}(S_{\omega,\rho}) \quad \text{and} \quad \frac{\partial^4 w}{\partial r^3 \partial \theta} \in L^p_{\gamma_3+1}(S_{\omega,\rho}).$$
Furthermore, we have
$$\begin{array}{lll}
\dis\frac{\partial^2 V_2}{\partial t^2}(t,\theta) & = & \dis \left(\rho e^{-t} \right)^3 \frac{\partial^4 v}{\partial r^4} (\rho e^{-t}, \theta) + 2\left(\rho e^{-t}\right)^2 \frac{\partial^3 v}{\partial r^3} (\rho e^{-t}, \theta) + \rho e^{-t} \frac{\partial^2 v}{\partial r^2} (\rho e^{-t}, \theta) \\ \ecart
&& \dis - \frac{\partial v}{\partial r} (\rho e^{-t}, \theta) - \frac{v(\rho e^{-t}, \theta)}{\rho e^{-t}}.
\end{array}$$
Since
$$\begin{array}{lll}
\dis \left(\rho e^{-t} \right)^3 \frac{\partial^4 v}{\partial r^4} (\rho e^{-t}, \theta) & = & \dis - 2\left(\rho e^{-t}\right)^2 \frac{\partial^3 v}{\partial r^3} (\rho e^{-t}, \theta) - \rho e^{-t} \frac{\partial^2 v}{\partial r^2} (\rho e^{-t}, \theta) \\ \ecart
&& \dis + \frac{\partial v}{\partial r} (\rho e^{-t}, \theta) + \frac{v(\rho e^{-t}, \theta)}{\rho e^{-t}} + \frac{\partial^2 V_2}{\partial t^2}(t,\theta),
\end{array}$$
and
$$\begin{array}{lll}
\dis\int_0^{+\infty} \int_0^\omega \left| \rho e^{-t} \frac{\partial^2 v}{\partial r^2} (\rho e^{-t}, \theta) \right|^p d\theta~dt & = &\dis\int_0^{\rho} \int_0^\omega \frac{r^{3p - 1}}{r}r^{-3p+1}\left|r \frac{\partial^2 v}{\partial r^2}(r,\theta)\right|^p d\theta~dr \\ \ecart

& = & \dis \int_0^\rho \int_0^\omega r^{3p - 2} r^{-3p+1}\left|r \frac{\partial^2 v}{\partial r^2}(r,\theta)\right|^p d\theta~dr \\ \ecart

& \leqslant & \dis \rho^{3p - 2} \int_0^\rho \int_0^\omega  r^{-3p+1}\left|r \frac{\partial^2 v}{\partial r^2}(r,\theta)\right|^p d\theta~dr < + \infty,
\end{array}$$
we deduce that
$$\int_0^{+\infty} \int_0^\omega \left| \left(\rho e^{-t}\right)^3 \frac{\partial^4 v}{\partial r^4} (\rho e^{-t}, \theta) \right|^p d\theta~dt = \int_0^{\rho} \int_0^\omega r^{3p - 1}\left|\frac{\partial^4 v}{\partial r^4}(r,\theta)\right|^p d\theta~dr < + \infty.$$
Therefore
$$\frac{\partial^4 v}{\partial r^4} \in L^p_{\gamma_4}(S_{\omega,\rho}) \quad \text{and} \quad \frac{\partial^4 w}{\partial r^4} \in L^p_{\gamma_4+1}(S_{\omega,\rho}),$$
where $\gamma_4 = 3 - \dfrac{1}{p}$.

It is clear now that the appropriated weight function, cited in \eqref{Def Wmp}, is $r \longmapsto r^\gamma := r^{3-\frac{1}{p}}$.

\section*{Acknowledgments} 

We would like to thank the referee for its valuable comments and corrections which have helped us to improve this paper.
 
\section*{Conflict of interest}

On behalf of all authors, the corresponding author states that there is no conflict of interest.

\end{document}